\documentclass[final]{siamltex}
\usepackage{amsmath,amssymb}
\usepackage{mathtools}

\usepackage{graphicx}
\usepackage{url}
\graphicspath{{./}}
\usepackage{float}
\usepackage{placeins}
\usepackage{tikz}
\usetikzlibrary{arrows.meta,positioning}
\newcommand{\dd}{\mathrm{d}}
\newcommand{\Z}{\mathbb{Z}}
\newcommand{\R}{\mathbb{R}}
\newcommand{\Prob}{\mathbb{P}}

\author{Joshua C.\ Chang\thanks{NIH Clinical Center, Rehabilitation Medicine Department (\texttt{josh.chang@nih.gov}); and Mederrata Research Inc.\ (\texttt{josh@mederrata.org}).}}
\title{Operator splitting for exploiting linear-rate closure in solving infinite ODE hierarchies}
\begin{document}
\maketitle

\begin{abstract}
    We introduce an operator-splitting method for infinite hierarchies of linear ordinary differential equations (ODEs) indexed by nonnegative integers.
    When the coupling coefficients depend linearly on the count index, an exact transformation closes the equations on finite count-index windows without an upper-boundary value.
    For more general hierarchies, Strang splitting applies the linear-rate closure during the linear-rate substeps and a conventional capped solver to the remainder.
    We derive the closure from generating functions and the method of characteristics and extend it to multi-indexed systems.
    The derivation requires neither positivity nor mass conservation, so it applies to a wider class of systems than the examplar stochastic models presented here.
    We discuss branching processes, stochastic predator-prey dynamics, the Schl\"ogl chemical kinetics model, and a telegraph model for gene expression.
    Through numerical experiments and computational cost analyses we demonstrate that our operator splitting method is typically advantageous for solving large scale systems in terms of memory usage and computational time, while retaining accuracy competitive with finite state projection (FSP) methods.
\end{abstract}



\section{Introduction}\label{sec:intro}

Master equations for continuous-time Markov processes form infinite ODE hierarchies.
For example, let $X_t$ denote the population count in a birth-death process at time $t$.
The probabilities $p_n(t)=\Prob(X_t=n)$ satisfy the conservation rule
\begin{equation}
    \frac{\dd p_n(t)}{\dd t} = \lambda_{n-1}\, p_{n-1}(t) + \mu_{n+1}\, p_{n+1}(t) - (\lambda_n + \mu_n)\, p_n(t),
    \qquad n\geq0,
    \label{eq:chapman}
\end{equation}
where $\lambda_n$ and $\mu_n$ are the birth and death rates, $\mu_0=0$, and we set $\lambda_{-1}p_{-1}(t):=0$.
Writing $p(t)=(p_n(t))_{n\geq0}$, let $\mathcal L$ denote the coefficient operator associated with Eq.~\ref{eq:chapman}, so that $\dot p(t)=\mathcal Lp(t)$.
For a Markov hierarchy, $\mathcal L$ is the forward generator under the column-vector convention --the transpose of the backward generator acting on observables---so its off-diagonal entries are nonnegative and its columns sum to zero.
We call a finite restriction of $\mathcal L$ a \emph{capped coefficient matrix}.
When the restriction retains escape rates, it is a subgenerator whose column sums are nonpositive and are negative for states with transitions out of the retained state space.
A conventional numerical solution strategy truncates the system at a nonnegative integer cap $N$, dropping the reflux from $p_{N+1}$.
The truncated solution can be approximated quickly using matrix exponential methods, but inaccuracy accumulates as the tail mass grows.
Increasing $N$, either in advance or during integration, becomes expensive in time and memory as the population grows.
The computational cost grows faster for multitype and hidden-state models because their capped coefficient matrices are larger and often cannot be exponentiated or factored efficiently.

Truncation is unnecessary when the transition-rate coefficients depend at most linearly on the index.
For linear-rate coefficients, we transform the original open system into an equivalent nonlinear closed system that determines $p_n$ for all $n\leq N$ without $p_{N+1}$.

For more general hierarchies, we split the coefficient operator by rate structure.
Closure advances the linear-rate part, and a capped sparse solver advances the nonlinear-rate remainder.
Splitting forgoes matrix exponentiation of the linear part but reduces memory use and can be faster for large $N$.

\paragraph{Prior methods}
Many numerical methods for count-indexed hierarchies use capping.
Finite-state projection (FSP) caps the coefficient operator but retains escape rates, so the missing mass bounds the truncation error~\cite{munskyFiniteStateProjection2006}.
The capped coefficient matrix can then be propagated by a Krylov approximation to the action of the matrix exponential~\cite{saadAnalysisKrylovSubspace1992} or by dense Pad\'e scaling-and-squaring for the full matrix exponential~\cite{highamScalingSquaringMethod2009}.
Uniformization instead evaluates the propagator through a Poisson-weighted matrix-power series~\cite{MarkoffChainsAid}.
For multivariate models, tensor trains compress the Cartesian product of count states~\cite{kazeevDirectSolutionChemical2014,dolgovSimultaneousStatetimeApproximation2015,gelssSolvingMasterEquation2016}.
For scalar birth-death models, Gaver--Stehfest inversion of a Laplace-domain continued fraction is another option~\cite{crawfordTransitionProbabilitiesGeneral2012,crawfordComputationalMethodsBirthdeath2018}.
Xia and Chou developed a kinetic theory that couples demographic stochasticity to continuous cell age and size~\cite{xiaKineticTheoryStructured2021} and later extended the framework to internal state and cell generation~\cite{xiaKineticTheoriesState2024}; these PDE and partial integrodifferential equation formulations are complementary to the count-indexed ODE closure studied here.
Veerman et al.~\cite{veermanTimedependentPropagatorsStochastic2018} also evolve the generating function along characteristics, but they recover coefficients by a contour fast Fourier transform (FFT) rather than the closed lower-triangular ODE whose dependency pattern is illustrated in Figure~\ref{fig:dependency_structure}(b).
Existing operator-splitting methods commonly separate a Markov generator by timescale, reaction, or species~\cite{macnamaraMultiscaleModelingChemical2008,choiStochasticOperatorsplittingMethod2012}.
We instead split by rate structure and solve the linear-rate subproblem by closure.

\section{Methods}\label{sec:methods}

Section~\ref{sec:closure} derives the finite closure.
Section~\ref{sec:splitting} embeds the closure in operator splitting.
Section~\ref{sec:applications} applies the method to branching processes, the Schl\"ogl chemical kinetics model, predator-prey kinetics, and a telegraph model for gene expression.

\subsection{Closure}\label{sec:closure}

\begin{definition}[Linear-rate ODE hierarchy]\label{def:source_affine}
    Consider a sequence $p(t)=(p_n(t))_{n\geq0}$ satisfying the initial-value problem $\dot p_n(t)=\sum_{m\geq0}\mathcal{L}_{n,m}p_m(t)$ with $p_n(0)=p_n^0$ for $n\geq0$.
    The coefficient operator $\mathcal{L}$ is \emph{linear-rate} if there exist finitely supported real sequences $\{\alpha_r\}_{r \geq -1}$ and $\{\beta_r\}_{r \geq 0}$ (with $\beta_{-1} := 0$) such that
    \begin{equation}
        \mathcal{L}_{n + r,\, n} = \alpha_r\, n + \beta_r,
        \qquad n\geq0,\quad n+r\geq0.
        \label{eq:source_affine}
    \end{equation}
\end{definition}
In the entry $\mathcal{L}_{n+r,n}$, $n$ is the source (column) index, $n+r$ is the destination (row) index, and the shift $r$ is their signed offset, $r=(n+r)-n$.
For each shift $r$, the constants $\alpha_r$ and $\beta_r$ are independent of the source index $n$; Eq.~\ref{eq:source_affine} displays all dependence on $n$.
As a concrete example, in the birth-death process of Eq.~\ref{eq:chapman}, let $\lambda_n=\lambda n+\nu$ and $\mu_n=\mu n$, where $\lambda,\mu\geq0$ are the per-particle birth and death rates and $\nu\geq0$ is the immigration rate.
Thus $\alpha_1=\lambda$, $\beta_1=\nu$, $\alpha_{-1}=\mu$, $\beta_{-1}=0$, $\alpha_0=-(\lambda+\mu)$, and $\beta_0=-\nu$.
Note that definition~\ref{def:source_affine} is agnostic to the sign of coefficients.

For a linear-rate ODE hierarchy, define the generating function and its initial value by
\begin{equation}
    G(z,t):=\sum_{n=0}^{\infty}p_n(t)z^n,
    \qquad
    G_0(z):=G(z,0)=\sum_{n=0}^{\infty}p_n^0z^n.
    \label{eq:scalar_generating_function}
\end{equation}
We interpret the series as analytic functions on their domains of convergence and use formal power series when extracting only finitely many coefficients.
The generating function satisfies the partial differential equation (PDE)
\begin{equation}
    \partial_t G(z, t) = A(z)\, \partial_z G(z, t) + B(z)\, G(z, t),
    \label{eq:first_order_pde}
\end{equation}
where $A(z) = \sum_{r\geq-1} \alpha_r\, z^{r + 1}$ and $B(z) = \sum_{r\geq0} \beta_r\, z^r$.

Eq.~\ref{eq:first_order_pde} admits a solution by the method of characteristics.
Let $\Phi_t(z)$ solve the characteristic ODE
\begin{equation}
    \frac{\dd \Phi_t(z)}{\dd t} = A(\Phi_t(z)), \qquad \Phi_0(z) = z,
    \label{eq:characteristic}
\end{equation}
and let the multiplier $K_t(z)$ solve
\begin{equation}
    \frac{\dd K_t(z)}{\dd t} = B(\Phi_t(z))\, K_t(z), \qquad K_0(z) = 1.
    \label{eq:multiplier}
\end{equation}
Power-series expansion of Eqs.~\ref{eq:characteristic} and~\ref{eq:multiplier} gives a lower-triangular coefficient recursion; Figure~\ref{fig:dependency_structure}(b) illustrates this dependency pattern for the pure birth-death example.

\begin{theorem}[Composition-multiplier representation]\label{thm:cm}
    Let an ODE hierarchy be linear-rate, with $\Phi_t$ and $K_t$ defined by Eqs.~\ref{eq:characteristic} and~\ref{eq:multiplier}, and generating function defined as in Eq.~\ref{eq:scalar_generating_function}.
    Let $\Omega\subseteq\mathbb{C}$ be a domain on which $G_0$ is defined, and let $I\subseteq[0,\infty)$ be a time interval containing zero such that $\Phi_t$ and $K_t$ are defined on $\Omega$ and $\Phi_t(\Omega)\subseteq\Omega$ for every $t\in I$.
    Then the unique classical solution of Eq.~\ref{eq:first_order_pde} on $\Omega\times I$ with initial value $G(z,0)=G_0(z)$ is
    \begin{equation}
        G(z, t) = K_t(z)\, G_0\bigl(\Phi_t(z)\bigr).
        \label{eq:composition_multiplier}
    \end{equation}
\end{theorem}
Setting $t=0$ in Eq.~\ref{eq:composition_multiplier} and using $\Phi_0(z)=z$ gives $G(z,0)=K_0(z)G_0(z)$.
Requiring this identity to recover the prescribed initial value $G(z,0)=G_0(z)$ for every $G_0$ fixes the normalization $K_0(z)=1$.
The proof follows by the method of characteristics~\cite{PartialDifferentialEquations}.
\begin{proof}
    Fix $t\in I$ and define the backward characteristic
    \begin{equation}
        \chi(s):=\Phi_{t-s}(z),\qquad 0\leq s\leq t.
        \label{eq:backward_characteristic}
    \end{equation}
    Eq.~\ref{eq:characteristic} gives $\dot\chi(s)=-A(\chi(s))$, $\chi(0)=\Phi_t(z)$, and $\chi(t)=z$.
    Applying the chain rule and Eq.~\ref{eq:first_order_pde} along Eq.~\ref{eq:backward_characteristic} gives
    \begin{equation}
        \frac{\dd}{\dd s}G(\chi(s),s)
        =\partial_tG(\chi(s),s)-A(\chi(s))\partial_zG(\chi(s),s)
        =B(\chi(s))G(\chi(s),s).
        \label{eq:solution_along_characteristic}
    \end{equation}
    Integrating Eq.~\ref{eq:solution_along_characteristic} from $0$ to $t$, changing variables $u=t-s$, and using $K_0(z)=1$ in Eq.~\ref{eq:multiplier} gives Eq.~\ref{eq:composition_multiplier}.
    Every solution of Eq.~\ref{eq:first_order_pde} satisfies the scalar linear ODE in Eq.~\ref{eq:solution_along_characteristic} along each characteristic in Eq.~\ref{eq:backward_characteristic}, so uniqueness for that ODE proves uniqueness of the representation.
\end{proof}

Relating to probability hierarchies, Eq.~\ref{eq:composition_multiplier} applies to the Kendall-Harris formula for branching with immigration~\cite{kendallGeneralizedBirthandDeathProcess1948,harrisBranchingProcesses1948}.
Eq.~\ref{eq:composition_multiplier} also applies to the monomolecular chemical master equation solution of Jahnke and Huisinga~\cite{jahnkeSolvingChemicalMaster2007}.

Finite-dimensional closure immediately follows by extracting power-series coefficients with the operator $[z^n]$ defined by
\begin{equation}
    [z^n]\sum_{k=0}^{\infty}h_kz^k:=h_n,\qquad n\geq0.
    \label{eq:coefficient_extraction}
\end{equation}

\begin{corollary}[Closed coefficient recursion]\label{cor:algebraic_recursion}
    Expand $\Phi_t(z) = \sum_{k=0}^{\infty} \phi_k(t)\, z^k$ and $K_t(z) = \sum_{k=0}^{\infty} \kappa_k(t)\, z^k$ as formal power series.
    Their coefficients satisfy
    \begin{equation}
        \begin{aligned}
            \dot \phi_n(t) & =[z^n]A(\Phi_t(z)),
                           & \dot \kappa_n(t)    & =[z^n]\bigl(B(\Phi_t(z))K_t(z)\bigr), \\
            \phi_n(0)      & =\delta_{n,1},
                           & \kappa_n(0)         & =\delta_{n,0}
        \end{aligned}
        \qquad n\geq0.
        \label{eq:coefficient_ode}
    \end{equation}
    Here $\delta_{n,m}$ is the Kronecker delta.
    At index $n$, the right-hand sides depend only on $\phi_0,\ldots,\phi_n$ and $\kappa_0,\ldots,\kappa_n$.
    Thus the equations through index $N$ form a closed ODE system and require no boundary value at $N+1$.
\end{corollary}

\begin{proof}
    Coefficientwise differentiation and application of the operator in Eq.~\ref{eq:coefficient_extraction} to Eqs.~\ref{eq:characteristic} and~\ref{eq:multiplier} give the differential equations in Eq.~\ref{eq:coefficient_ode}.
    The initial values follow from $\Phi_0(z)=z$ and $K_0(z)=1$ in Eqs.~\ref{eq:characteristic} and~\ref{eq:multiplier}.
    Since $A(w)=\sum_{r\geq-1}\alpha_rw^{r+1}$, the finite Cauchy products give~\cite{flajoletAnalyticCombinatorics2009,pivoteauAlgorithmsCombinatorialSystems2012}
    \begin{equation}
        [z^n]A(\Phi_t(z))
        =\alpha_{-1}\delta_{n,0}+\sum_{r\geq0}\alpha_r
        \sum_{\substack{k_1+\cdots+k_{r+1}=n \\k_1,\ldots,k_{r+1}\geq0}}
        \prod_{i=1}^{r+1}\phi_{k_i}(t).
        \label{eq:A_coefficient_expansion}
    \end{equation}
    Likewise, $B(w)=\sum_{r\geq0}\beta_rw^r$ and the product with $K_t$ give
    \begin{equation}
        [z^n]\bigl(B(\Phi_t(z))K_t(z)\bigr)
        =\beta_0\kappa_n(t)+\sum_{r\geq1}\beta_r
        \sum_{\substack{j+k_1+\cdots+k_r=n \\j,k_1,\ldots,k_r\geq0}}
        \kappa_j(t)\prod_{i=1}^{r}\phi_{k_i}(t).
        \label{eq:BK_coefficient_expansion}
    \end{equation}
    Every index in the sums of Eqs.~\ref{eq:A_coefficient_expansion} and~\ref{eq:BK_coefficient_expansion} is at most $n$.
    Thus the equations for index $n$ use only $\phi_0,\ldots,\phi_n$ and $\kappa_0,\ldots,\kappa_n$, and finite support of $\{\alpha_r\}$ and $\{\beta_r\}$ makes both outer sums finite.
\end{proof}

\paragraph{Birth-death example}
For pure birth-death with per-particle birth and death rates $\lambda,\mu\geq0$, we have $\alpha_1=\lambda$, $\alpha_{-1}=\mu$, $\alpha_0=-(\lambda+\mu)$, and $\beta_r=0$ for every $r$.
Since $\beta_r=0$ for every $r$, $B(z)\equiv0$.
Eq.~\ref{eq:multiplier} and $K_0(z)=1$ therefore give $K_t(z)\equiv1$.
With one initial particle, $G_0(z)=z$, so Eq.~\ref{eq:composition_multiplier} reduces to $G(z,t)=\Phi_t(z)$.
We therefore write $\Phi_t(z)=\sum_{n=0}^{\infty}p_n(t)z^n$.
Eq.~\ref{eq:coefficient_ode} becomes
\begin{equation}
    \dot p_n(t) = -(\lambda + \mu)\, p_n(t) + \mu\, \delta_{n,0} + \lambda \sum_{k = 0}^n p_k(t)\, p_{n - k}(t),
    \qquad n\geq0.
    \label{eq:bd_riccati}
\end{equation}
Figure~\ref{fig:dependency_structure} compares the dependency graphs in Eqs.~\ref{eq:chapman} and~\ref{eq:bd_riccati}.
\begin{figure}[!ht]
    \centering
    \includegraphics[width=0.96\textwidth]{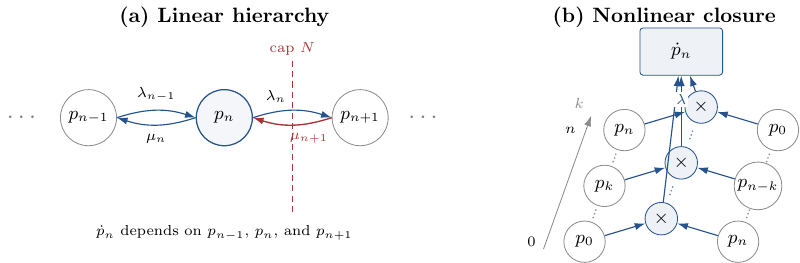}
    \caption{Dependency graphs for the original linear hierarchy and nonlinear closure.
        \textbf{(a)} Eq.~\ref{eq:chapman} gives $\dot p_n$ from $p_{n-1}$, $p_n$, and $p_{n+1}$; when $n=N$, the red arrow marks return flux from $p_{N+1}$ across the cap.
        \textbf{(b)} The closure removes the dependence on $p_{n+1}$ and introduces a quadratic convolution.
        At depth $k$, the nodes $p_k$ and $p_{n-k}$ feed a multiplication node whose output contributes $\lambda p_kp_{n-k}$ to $\dot p_n$.
        Since $0\leq k\leq n$, no coefficient with index greater than $n$ enters the nonlinear equation.
        Panel \textbf{(b)} omits the unchanged linear term $-(\lambda+\mu)p_n$.}
    \label{fig:dependency_structure}
\end{figure}

In the pure birth-death case, Eq.~\ref{eq:characteristic} is the scalar Riccati equation $\dot\Phi=\lambda\Phi^2-(\lambda+\mu)\Phi+\mu$.
Solving the Riccati equation and evaluating at $z=0$, where $p_0(t)=\Phi_t(0)$, gives $p_0(t)=(\mu-\mu e^{-(\mu-\lambda)t})/(\mu-\lambda e^{-(\mu-\lambda)t})$ for $\lambda\neq\mu$ and $p_0(t)=\mu t/(1+\mu t)$ in the critical limit.
For numerical stability when $\mu>0$, use the transformation $\rho(t):=(\lambda/\mu)p_0(t)$ and $p_1(t):=(1-p_0(t))(1-\rho(t))$.
The transformation gives the geometric tail
\begin{equation}
    p_n(t) = p_1(t)\, \rho(t)^{n - 1}, \qquad n \geq 1.
    \label{eq:bd_geometric_tail}
\end{equation}
The geometric-tail representation in Eq.~\ref{eq:bd_geometric_tail} is less susceptible to floating-point cancellation than direct coefficient extraction from the linear-fractional generating function.

A direct solver of Eq.~\ref{eq:chapman} on $0\leq n\leq N$ must close the upper boundary, commonly by setting $p_{N+1}=0$.
The zero-boundary approximation removes the return-flux term $\mu_{N+1}p_{N+1}$ from the equation for $p_N$.

Assume Definition~\ref{def:source_affine}, and let $r_{\max}\geq0$ be an upper bound on the supports of both coefficient sequences.
Let $[0,T_\star)$ be the maximal time interval on which coefficientwise smooth formal power series $\Phi_t,K_t\in\R[[z]]$ solve Eqs.~\ref{eq:characteristic} and~\ref{eq:multiplier}.
We call a finite set of retained count indices an \emph{index window}; time domains such as $[0,T_\star)$ are instead called intervals.
For the index window $\{0,\ldots,N\}$, define the truncated polynomials
\begin{equation}
    \Phi_t^{(N)}(z):=\sum_{n=0}^N\phi_n^{(N)}(t)z^n,
    \qquad
    K_t^{(N)}(z):=\sum_{n=0}^N\kappa_n^{(N)}(t)z^n.
    \label{eq:truncated_characteristic_series}
\end{equation}
Using the polynomials in Eq.~\ref{eq:truncated_characteristic_series}, define the finite initial-value problem
\begin{equation}
    \begin{aligned}
        \dot\phi_n^{(N)}(t)
         & =[z^n]A\bigl(\Phi_t^{(N)}(z)\bigr),
         & \phi_n^{(N)}(0)                                   & =\delta_{n,1}, \\
        \dot\kappa_n^{(N)}(t)
         & =[z^n]\bigl(B(\Phi_t^{(N)}(z))K_t^{(N)}(z)\bigr),
         & \kappa_n^{(N)}(0)                                 & =\delta_{n,0},
    \end{aligned}
    \qquad 0\leq n\leq N.
    \label{eq:coefficient_ode_truncated}
\end{equation}
We evaluate each right-hand side from the finite Cauchy products involving coefficients with indices at most $n$.

\begin{proposition}[Exactness on a finite index window]\label{prop:in_window}
    The finite system in Eq.~\ref{eq:coefficient_ode_truncated} has a unique maximal classical solution on an interval $[0,T_N)$ with $T_N\geq T_\star$.
    On $[0,T_\star)$, that solution satisfies
    \begin{equation}
        \phi_n^{(N)}(t) = [z^n]\, \Phi_t(z),
        \qquad \kappa_n^{(N)}(t) = [z^n]\, K_t(z),
        \qquad 0 \leq n \leq N,
        \label{eq:in_window_exact}
    \end{equation}
    and Eq.~\ref{eq:coefficient_ode_truncated} contains no coefficient with index greater than $N$.
\end{proposition}

We prove the proposition by establishing the general triangular dependency structure illustrated for birth-death in Figure~\ref{fig:dependency_structure}(b), identifying the coefficients at each index by induction, and discarding coefficients at higher indices.
\begin{proof}
    \emph{Step 1: triangular structure.}
    Eqs.~\ref{eq:A_coefficient_expansion} and~\ref{eq:BK_coefficient_expansion} show that the right-hand side for index $n$ contains no coefficient with index greater than $n$.
    Thus $(\dot\phi_n^{(N)},\dot\kappa_n^{(N)})=R_n(\phi_0^{(N)},\ldots,\phi_n^{(N)},\kappa_0^{(N)},\ldots,\kappa_n^{(N)})$ for a polynomial map $R_n:\R^{2(n+1)}\to\R^2$ of degree at most $r_{\max}+1$.

    \emph{Step 2: induction on the index.}
    We prove equality with the formal-series solution by induction on $n$.
    At $n=0$, the polynomial map $R_0$ is locally Lipschitz continuous.
    The Picard-Lindel\"of theorem therefore gives a unique local solution from $\phi_0^{(N)}(0)=0$ and $\kappa_0^{(N)}(0)=1$.
    The pair $([z^0]\Phi_t(z),[z^0]K_t(z))$ satisfies the same initial-value problem, so the two solutions agree wherever both are defined.
    For the inductive step, suppose the assertion holds for indices $0, \ldots, n-1$.
    We substitute the already identified coefficients with indices $0, \ldots, n-1$ into the component for index $n$, leaving a polynomial ODE in $(\phi_n^{(N)},\kappa_n^{(N)})\in\R^2$ whose coefficients depend on time through the lower-index solutions.
    Its right-hand side is continuous in time and locally Lipschitz continuous in $(\phi_n^{(N)},\kappa_n^{(N)})$.
    Thus the hypotheses of the Picard-Lindel\"of theorem are satisfied, giving a unique local solution from $(\delta_{n,1},\delta_{n,0})$.
    The pair $([z^n]\Phi_t(z),[z^n]K_t(z))$ satisfies that ODE with the same initial values, so uniqueness completes the induction.

    \emph{Step 3: independence from higher indices.}
    Step~1 shows that every $R_n$ with $n\leq N$ ignores all coefficients with indices greater than $N$.
    The truncated system therefore closes on $\R^{2(N+1)}$ without introducing any remaining formal coefficients.
    Step~2 and the extension theorem for finite-dimensional polynomial ODEs give the maximal solution on $[0,T_N)$ with $T_N\geq T_\star$.
\end{proof}

\paragraph{Reconstruction of $\{p_n(t)\}$ and computational cost}
We reconstruct the probability distribution via
\[
    p_n(t)=\sum_{m\geq0}p_m^0[z^n]\bigl(K_t(z)\Phi_t(z)^m\bigr)
\]
using iterated Cauchy products.
For initial data with infinite support, let $M_0\geq0$ be the largest retained initial index.
Restricting the reconstruction sum to $0\leq m\leq M_0$ introduces an initial-tail error but no boundary condition at $N+1$.
With $S$ coefficient-ODE evaluations, direct Cauchy products require $\Theta(SN^2)$ operations for the ODE and $\Theta(N^2M_0)$ operations for reconstruction.
The computational cost of reconstruction falls to $\Theta(N^2)$ for a single initial component.

\subsubsection{Multivariate closure}\label{sec:scalar_multitype}

Let the positive integer $K$ denote the number of count axes, and let $\boldsymbol{n}=(n_1,\ldots,n_K)\in\Z_{\geq0}^K$ and $\boldsymbol{z}=(z_1,\ldots,z_K)$.
Define the multivariate generating function, its initial value, and its monomials by
\begin{equation}
    F(\boldsymbol{z}, t)
    :=\sum_{\boldsymbol{n} \in \Z_{\geq 0}^K} p_{\boldsymbol{n}}(t)\,\boldsymbol{z}^{\boldsymbol{n}},
    \qquad
    F_0(\boldsymbol{z}):=F(\boldsymbol{z},0),
    \qquad
    \boldsymbol{z}^{\boldsymbol{n}}:=\prod_{i=1}^K z_i^{n_i}.
    \label{eq:multitype_generating_function}
\end{equation}
For a source multi-index $\boldsymbol{n}$, a shift $\boldsymbol{r}\in\Z^K$ is the componentwise signed offset to the destination multi-index $\boldsymbol{n}+\boldsymbol{r}$, so $\boldsymbol{r}=(\boldsymbol{n}+\boldsymbol{r})-\boldsymbol{n}$.
Assume that $\mathcal{R}\subseteq\Z^K$ is a finite set of such shifts with real coefficients $\alpha^{(i)}_{\boldsymbol{r}}$ and $\beta_{\boldsymbol{r}}$ such that
\begin{equation}
    \mathcal{L}_{\boldsymbol{n}+\boldsymbol{r},\boldsymbol{n}}
    =\sum_{i=1}^K\alpha^{(i)}_{\boldsymbol{r}}n_i+\beta_{\boldsymbol{r}},
    \qquad
    \boldsymbol{r}\in\mathcal{R},\quad
    \boldsymbol{n},\boldsymbol{n}+\boldsymbol{r}\in\Z_{\geq0}^K.
    \label{eq:multitype_source_affine}
\end{equation}
To keep the generating-function coefficients polynomial, we require $z_i\boldsymbol{z}^{\boldsymbol r}$ when $\alpha^{(i)}_{\boldsymbol r}\neq0$ and $\boldsymbol{z}^{\boldsymbol r}$ when $\beta_{\boldsymbol r}\neq0$ to contain no negative powers.
All of the applications in this manuscript satisfy these conditions.

The generating function satisfies the multivariate first-order PDE
\begin{equation}
    \partial_t F = \sum_{i = 1}^K A_i(\boldsymbol{z})\, \partial_{z_i} F + B(\boldsymbol{z})\, F,
    \label{eq:multitype_pde}
\end{equation}
with $A_i(\boldsymbol{z}) = \sum_{\boldsymbol{r}\in\mathcal{R}} \alpha^{(i)}_{\boldsymbol{r}}\, z_i\, \boldsymbol{z}^{\boldsymbol{r}}$ and $B(\boldsymbol{z}) = \sum_{\boldsymbol{r}\in\mathcal{R}} \beta_{\boldsymbol{r}}\, \boldsymbol{z}^{\boldsymbol{r}}$.
Let $\boldsymbol{A}:=(A_1,\ldots,A_K)$, and define the vector characteristic $\boldsymbol{\Phi}_t$ and scalar multiplier $K_t$ by
\begin{equation}
    \begin{aligned}
        \frac{\dd \boldsymbol{\Phi}_t(\boldsymbol{z})}{\dd t}
         & =\boldsymbol{A}\bigl(\boldsymbol{\Phi}_t(\boldsymbol{z})\bigr),
         & \boldsymbol{\Phi}_0(\boldsymbol{z})                                   & =\boldsymbol{z}, \\
        \frac{\dd K_t(\boldsymbol{z})}{\dd t}
         & =B\bigl(\boldsymbol{\Phi}_t(\boldsymbol{z})\bigr)K_t(\boldsymbol{z}),
         & K_0(\boldsymbol{z})                                                   & =1.
    \end{aligned}
    \label{eq:multitype_characteristic}
\end{equation}
On any domain where the characteristic, multiplier, and composition are defined, the method of characteristics gives
\begin{equation}
    F(\boldsymbol{z}, t) = K_t(\boldsymbol{z})\, F_0\bigl(\boldsymbol{\Phi}_t(\boldsymbol{z})\bigr).
    \label{eq:multitype_composition_multiplier}
\end{equation}
Setting $t=0$ in Eq.~\ref{eq:multitype_composition_multiplier} and using $\boldsymbol{\Phi}_0(\boldsymbol{z})=\boldsymbol{z}$ gives $F(\boldsymbol{z},0)=K_0(\boldsymbol{z})F_0(\boldsymbol{z})$.
Requiring this identity to recover the prescribed initial value $F(\boldsymbol{z},0)=F_0(\boldsymbol{z})$ for every $F_0$ fixes the normalization $K_0(\boldsymbol{z})=1$.
For multi-indices, $\boldsymbol{k}\leq\boldsymbol{n}$ means $k_i\leq n_i$ for every $i$.
Because multivariate Cauchy products respect componentwise order, the recursion at $\boldsymbol n$ uses only $\boldsymbol k\leq\boldsymbol n$, and Proposition~\ref{prop:in_window} extends to $0\leq n_i\leq N$ for every $i$.

Closure avoids the need to store the capped coefficient matrix on the product state space.
Instead, it stores one coefficient array for each of the $K$ characteristic components and one for the multiplier, all indexed by $\{0,\ldots,N\}^K$.
Direct Cauchy products sum over all componentwise decompositions of these multi-indices, so the arithmetic work still increases rapidly with $N$ and $K$, as the next proposition quantifies.
\begin{proposition}[Computational cost of multivariate closure]\label{prop:complexity}
    Consider a model family with $\Theta(K)$ nonzero rate terms, $K$ count axes, per-axis cap $N$, and joint dimension $D=(N+1)^K$.
    Let $S$ denote the number of coefficient-ODE right-hand-side evaluations.
    Let $s_{\mathrm P}\geq0$ denote the number of squarings used by dense Pad\'e scaling-and-squaring.
    On the capped joint coefficient matrix, this method requires $\Theta((1+s_{\mathrm P})D^3)=\Theta((1+s_{\mathrm P})(N+1)^{3K})$ floating-point operations and $\Theta(D^2)=\Theta((N+1)^{2K})$ memory~\cite{highamScalingSquaringMethod2009}.
    With direct Cauchy products, the multi-index closure uses $\Theta\!\left(SK\left((N+1)(N+2)/2\right)^K\right)=\Theta(SK(N+1)^{2K})$ floating-point operations for fixed $K$ and stores $\Theta(K (N+1)^K)$ coefficient entries.
    For fixed $K$ and $S$, the asymptotic ratio of floating-point operation counts as $N\to\infty$ is
    \begin{equation}
        \frac{\text{dense computational cost}}{\text{closure computational cost}} = \Theta\!\left(\frac{(1+s_{\mathrm P})(N+1)^K}{S\, K}\right),
        \label{eq:multitype_complexity_ratio}
    \end{equation}
    while the exact Cauchy-product count shows exponential dependence on $K$ at fixed $N$.
\end{proposition}
\begin{proof}
    Evaluating the bounded-degree Pad\'e approximant requires a fixed number of dense $D\times D$ matrix products and one dense linear solve, and each of the $s_{\mathrm P}$ subsequent squarings is another dense matrix product.
    A dense matrix product forms $D^2$ entries, each from a length-$D$ inner product, and therefore requires $\Theta(D^3)$ floating-point operations.
    A dense solve based on lower--upper (LU) factorization has the same cubic order, while a fixed number of $D\times D$ work arrays contains $\Theta(D^2)$ entries~\cite{highamScalingSquaringMethod2009}.
    The scaling parameter $s_{\mathrm P}$ grows at most logarithmically with the norm of the matrix argument.
    Each closure right-hand side contains $\Theta(K)$ Cauchy products; on $\{0,\ldots,N\}^K$ each has $\bigl((N+1)(N+2)/2\bigr)^K$ multiply-add terms.
    Storing the $KD$ coefficients requires $\Theta(KD)=\Theta(K(N+1)^K)$ memory.
    Multiplying the Cauchy-product work by $S$ gives $\Theta(SK(N+1)^{2K})$ floating-point operations.
    Taking the ratio and substituting $D=(N+1)^K$ gives Eq.~\ref{eq:multitype_complexity_ratio}.
    The value of $S$ depends on the tolerance and stiffness, so we hold it fixed only in the asymptotic comparison.
\end{proof}

\subsubsection{Finite multi-state extension}\label{sec:matrix}

Many applications track counts of residency in internal states, such as intermediate gene transcription states.
If every internal state has the same per-particle loss rate, then the count follows one scalar characteristic and only the multiplier becomes matrix-valued.

Let the positive integer $n_T$ denote the number of internal states.
For each count $m\geq0$, let the vector $\boldsymbol{P}_m(t)\in\R^{n_T}$ collect the corresponding components.
We consider
\begin{equation}
    \frac{\dd \boldsymbol{P}_m}{\dd t}
    =\mathbf{A}\boldsymbol{P}_m+\mathbf{B}\boldsymbol{P}_{m-1}
    +\mu\bigl((m+1)\boldsymbol{P}_{m+1}-m\boldsymbol{P}_m\bigr),
    \qquad m\geq0,\quad \boldsymbol{P}_{-1}\equiv0,
    \label{eq:matrix_count_master}
\end{equation}
where $\mathbf{A},\mathbf{B}\in\R^{n_T\times n_T}$ collect transitions that preserve and increment $m$, respectively, and $\mu\geq0$ is the per-object loss rate.

Define the vector-valued generating function $\boldsymbol{Z}(z,t):=\sum_{m=0}^{\infty}\boldsymbol{P}_m(t)z^m$ and its initial value $\boldsymbol{Z}_0(z):=\boldsymbol{Z}(z,0)$.
The closure acts componentwise on the internal-state vector: each component of $\boldsymbol{Z}$ is a scalar power series in $z$, while the matrices $\mathbf{A}$, $\mathbf{B}$, and $\mathbf{K}_t$ retain the coupling among components.
Multiplying Eq.~\ref{eq:matrix_count_master} by $z^m$ and summing gives the vector-valued first-order PDE
\begin{equation}
    \partial_t \boldsymbol{Z}(z, t) = (\mathbf{A} + z \mathbf{B})\, \boldsymbol{Z}(z, t) + \mu(1 - z)\, \partial_z \boldsymbol{Z}(z, t).
    \label{eq:matrix_count_pgf}
\end{equation}
Eq.~\ref{eq:matrix_count_pgf} has the same transport term as the scalar PDE, but its count-preserving and count-incrementing coefficients are matrices.

\begin{corollary}[Matrix recursion]\label{cor:matrix_recursion}
    For the hierarchy in Eq.~\ref{eq:matrix_count_master}, the scalar characteristic is
    \begin{equation}
        \Phi_t(z)=1-(1-z)e^{-\mu t}.
        \label{eq:matrix_characteristic}
    \end{equation}
    The solution has the representation
    \begin{equation}
        \boldsymbol{Z}(z, t) = \mathbf{K}_t(z)\, \boldsymbol{Z}_0\bigl(\Phi_t(z)\bigr),
        \label{eq:matrix_count_cm}
    \end{equation}
    where the matrix multiplier $\mathbf{K}_t(z) \in \mathbb{C}^{n_T \times n_T}$ solves
    \begin{equation}
        \frac{\dd \mathbf{K}_t(z)}{\dd t} = \mathbf{K}_t(z)\,\bigl(\mathbf{A} + \Phi_t(z)\, \mathbf{B}\bigr),\qquad \mathbf{K}_0(z) = \mathbf{I}.
        \label{eq:matrix_multiplier}
    \end{equation}
    Here $\mathbf{I}$ is the $n_T\times n_T$ identity matrix.
    Writing $\mathbf{K}_t(z)=\sum_{n=0}^{\infty}\mathbf{Q}_n(t)z^n$ with $\mathbf{Q}_n(t)\in\R^{n_T\times n_T}$, the matrix-valued power-series coefficients satisfy
    \begin{equation}
        \dot{\mathbf{Q}}_n(t)
        =\mathbf{Q}_n(t)\mathbf{A}
        +\sum_{\substack{j+k=n\\j,k\geq0}}\mathbf{Q}_j(t)\phi_k(t)\mathbf{B},
        \qquad
        \mathbf{Q}_n(0)=\delta_{n,0}\mathbf{I},
        \qquad n\geq0,
        \label{eq:matrix_coefficient_recursion}
    \end{equation}
    where $\phi_k(t)=[z^k]\Phi_t(z)$.
    Thus the coefficients through any index $M$ form a lower-triangular ODE with $n_T\times n_T$ blocks, the block analogue of the scalar structure in Figure~\ref{fig:dependency_structure}(b).
\end{corollary}

Because matrix multiplication does not commute, the characteristic cocycle determines the order in Eq.~\ref{eq:matrix_multiplier}.
\begin{proof}
    The transport coefficient in Eq.~\ref{eq:matrix_count_pgf} is $\mu(1-z)$, so Eq.~\ref{eq:matrix_characteristic} solves $\dot\Phi_t=\mu(1-\Phi_t)$ with $\Phi_0(z)=z$.
    Set $\mathbf{C}(z)=\mathbf{A}+z\mathbf{B}$ and follow the scalar characteristic $\Phi_t$.
    The right-ordered multiplier has the cocycle identity $\mathbf{K}_{t+s}(z)=\mathbf{K}_s(z)\mathbf{K}_t(\Phi_s(z))$.
    Differentiating the cocycle identity at $s=0$ gives
    \begin{equation}
        \mu(1-z)\partial_z\mathbf{K}_t(z)+\mathbf{C}(z)\mathbf{K}_t(z)
        =\mathbf{K}_t(z)\mathbf{C}(\Phi_t(z))
        =\partial_t\mathbf{K}_t(z).
        \label{eq:matrix_cocycle_derivative}
    \end{equation}
    The last equality in Eq.~\ref{eq:matrix_cocycle_derivative} is Eq.~\ref{eq:matrix_multiplier}.
    The cocycle identity verifies Eq.~\ref{eq:matrix_count_cm} for column vectors, and expansion of Eq.~\ref{eq:matrix_multiplier} gives Eq.~\ref{eq:matrix_coefficient_recursion}.
\end{proof}

Because the characteristic in Eq.~\ref{eq:matrix_characteristic} is affine in $z$, the equation for $\mathbf Q_n$ in Eq.~\ref{eq:matrix_coefficient_recursion} couples only $\mathbf Q_n$ and, for $n\geq1$, $\mathbf Q_{n-1}$.
For a nonnegative integer count cap $M$, integrating the coefficients of $\mathbf{K}_t$ through index $M$ therefore requires $\Theta(SMn_T^3)$ floating point operations and $\Theta(Mn_T^2)$ memory for $S$ right-hand-side evaluations, without a boundary condition at $M+1$.
Dense Pad\'e on the joint state of dimension $n_T(M+1)$ requires $\Theta(n_T^3 M^3)$ time and $\Theta(n_T^2 M^2)$ memory.
Relative to dense Pad\'e, closure saves a factor of order $M^2/S$ in time and $M$ in memory.
The shared loss rate is essential: state-specific loss rates produce different count characteristics and fall outside the scalar-drift construction~\cite{cuchieroAffineProcessesPositive2011}.

\subsubsection{Stationary closure}\label{sec:steady}

At stationarity, the full state vector $p$ satisfies $\mathcal{L}p=0$, and its components sum to one.
Setting $\partial_t\boldsymbol{Z}=0$ in Eq.~\ref{eq:matrix_count_pgf} gives a matrix ODE in $z$ for the vector-valued probability generating function (PGF).
A Frobenius expansion selects the solution regular at $z=1$, and a contour fast Fourier transform (FFT) recovers its coefficients; we refer to this procedure PGF-FFT.
Alternatively, setting Eq.~\ref{eq:matrix_count_master} to zero gives the block-tridiagonal recurrence
\begin{equation}
    \mu(m+1)\boldsymbol{P}_{m+1}-(\mu m\mathbf{I}-\mathbf{A})\boldsymbol{P}_m+\mathbf{B}\boldsymbol{P}_{m-1}=0,
    \qquad m\geq0.
    \label{eq:steady_recurrence}
\end{equation}
Block-Thomas elimination~\cite{neutsMatrixgeometricSolutionsStochastic1994} solves the capped recurrence in $\mathcal{O}(Mn_T^3)$ work; we propagate backward because forward iteration amplifies nonphysical growing solutions.
On a contour of radius $0<r_c<1$, PGF-FFT avoids a boundary condition at $M$ but amplifies roundoff by the Cauchy factor $r_c^{-m}$.
As a general-purpose baseline, sparse LU factorization solves the capped stationary system after one equation is replaced by the normalization constraint~\cite{davisDirectMethodsSparse2006}.
Dense SVD instead uses the right singular vector associated with the smallest singular value of the capped coefficient matrix~\cite{MatrixComputations4th}.

\subsubsection{Closure implementation}\label{sec:closure_implementation}

At a final time $T$, we first integrate Eq.~\ref{eq:coefficient_ode_truncated}, then construct $[z^n]\bigl(K_T^{(N)}(z)(\Phi_T^{(N)}(z))^m\bigr)$ for $m\leq M_0$ by Cauchy products and apply the resulting kernel to the initial data.
With $S$ right-hand-side evaluations, the computational costs of coefficient integration, kernel construction, and matrix-vector multiplication are $\Theta(SN^2)$, $\Theta(N^2M_0)$, and $\Theta(NM_0)$, respectively.
FFT convolution accelerates the scalar and multitype products in the experiments.

\subsection{Operator splitting}\label{sec:splitting}

Closure removes the upper-boundary approximation only when the complete hierarchy is linear-rate.
For a broader class of ODE hierarchies, we decompose the generator as
\begin{equation}
    \mathcal{L} = \mathcal{A} + \mathcal{B},
    \label{eq:hybrid_split}
\end{equation}
with a linear-rate part $\mathcal A$ and a nonlinear-rate remainder $\mathcal B$.
Let $\Pi_N$ denote projection onto the selected finite index window.
For one count axis, $\Pi_N$ sets components with $n>N$ to zero; for $K$ count axes, it sets components outside $\{0,\ldots,N\}^K$ to zero.
We advance the capped remainder $\mathcal B_N=\Pi_N\mathcal B\Pi_N$ with a sparse solver.
For a step size $\Delta t>0$, unprojected Strang splitting gives the second-order approximation
\begin{equation}
    \tilde p(t + \Delta t) = e^{(\Delta t/2)\,\mathcal{A}}\, e^{\Delta t\, \mathcal{B}}\, e^{(\Delta t/2)\,\mathcal{A}}\, p(t).
    \label{eq:strang}
\end{equation}

Let $A_{\mathcal A}$ and $B_{\mathcal A}$ denote the generating-function polynomials in Eq.~\ref{eq:first_order_pde} for the linear-rate coefficient operator $\mathcal A$.
Let $\Phi_\tau$ and $K_\tau$ denote the corresponding characteristic and multiplier from Eqs.~\ref{eq:characteristic} and~\ref{eq:multiplier}.
For $\tau\geq0$, we write $\mathcal{C}_{\tau,N}:=\Pi_Ne^{\tau\mathcal{A}}\Pi_N$ for the action on the finite index window computed by closure.
For one count axis and an input $q=(q_0,\ldots,q_N)$ supported on this index window, define $Q(z):=\sum_{m=0}^Nq_mz^m$.
The characteristic and multiplier retain the fixed initial values $\Phi_0(z)=z$ and $K_0(z)=1$ for every substep; the arbitrary input $q$ enters only through composition with $Q$:
\begin{equation}
    \bigl(\mathcal{C}_{\tau,N}q\bigr)_n
    =[z^n]\bigl(K_\tau(z)Q(\Phi_\tau(z))\bigr)
    =\sum_{m=0}^Nq_m[z^n]\bigl(K_\tau(z)\Phi_\tau(z)^m\bigr),
    \qquad 0\leq n\leq N.
    \label{eq:closure_arbitrary_input}
\end{equation}
The bracketed coefficients in Eq.~\ref{eq:closure_arbitrary_input} form a kernel that depends on $\mathcal A$, $\tau$, and $N$ but not on $q$.
For fixed $\Delta t$, we therefore build this kernel once at $\tau=\Delta t/2$ and apply it to the arbitrary intermediate state entering either linear-rate half-step.
The implemented capped Strang map is
\begin{equation}
    \mathcal{S}_{\Delta t,N}
    :=\mathcal{C}_{\Delta t/2,N}\,e^{\Delta t\mathcal{B}_N}\,\mathcal{C}_{\Delta t/2,N}.
    \label{eq:implemented_strang}
\end{equation}
For fixed $N$, let $p^j:=\mathcal{S}_{\Delta t,N}^j\Pi_Np(0)$ denote the numerical state after $j$ steps.
Figure~\ref{fig:closure_split} shows the closure calculation and one complete Strang step.

\begin{figure}[H]
    \centering
    \begin{tikzpicture}[
        font=\footnotesize,
        >={Latex[length=2mm]},
        flow/.style={->, semithick},
        box/.style={draw=black!70, fill=black!3, rounded corners=1pt, align=center, inner sep=3pt, minimum height=9mm},
        linear/.style={box, fill=blue!7},
        remainder/.style={box, fill=orange!10}
        ]
        \node[box, text width=3.0cm, anchor=west] (pgf) at (0,1.25) {$p^j=(p_0^j,\ldots,p_N^j)$\\$G^j(z)=\sum_{n=0}^N p_n^jz^n$};
        \node[font=\bfseries, anchor=south west] at ([yshift=1mm]pgf.north west) {Closure $\mathcal{C}_{\tau,N}$};
        \node[box, text width=4.0cm, right=4mm of pgf] (ode) {$\dot\Phi=A_{\mathcal A}(\Phi)$,\quad $\dot K=B_{\mathcal A}(\Phi)K$\\integrate coefficients $0{:}N$};
        \node[box, text width=3.6cm, right=4mm of ode] (extract) {$\displaystyle (\mathcal C_{\tau,N} p^j)_n=[z^n]\bigl(K_\tau G^j(\Phi_\tau)\bigr)$\\$0\leq n\leq N$};
        \draw[flow] (pgf) -- (ode);
        \draw[flow] (ode) -- (extract);
        \node[font=\scriptsize, below=1mm of ode] {No value at $n=N+1$ enters the coefficient ODE.};

        \node[box, anchor=west, minimum width=1.4cm] (p0) at (0,-1.35) {$p^j$};
        \node[font=\bfseries, anchor=south west] at ([yshift=1mm]p0.north west) {Strang step};
        \node[linear, right=3mm of p0, minimum width=2.2cm] (c1) {$\mathcal C_{\Delta t/2,N}$};
        \node[remainder, right=3mm of c1, minimum width=3.0cm] (bn) {$e^{\Delta t\mathcal B_N}$\\sparse solve};
        \node[linear, right=3mm of bn, minimum width=2.2cm] (c2) {$\mathcal C_{\Delta t/2,N}$};
        \node[box, right=3mm of c2, minimum width=1.5cm] (p1) {$p^{j+1}$};
        \draw[flow] (p0) -- (c1);
        \draw[flow] (c1) -- (bn);
        \draw[flow] (bn) -- (c2);
        \draw[flow] (c2) -- (p1);
        \node[font=\scriptsize, below=1mm of c1] {linear-rate};
        \node[font=\scriptsize, below=1mm of bn] {nonlinear-rate remainder};
        \node[font=\scriptsize, below=1mm of c2] {linear-rate};
    \end{tikzpicture}
    \caption{How coefficient closure enters one operator-splitting step.
        The top row computes the requested coefficients of the linear-rate propagator without prescribing a value at $N+1$.
        The bottom row uses the linear-rate propagator for both half-steps and projects onto the capped index window used by the sparse nonlinear-rate middle step.}
    \label{fig:closure_split}
\end{figure}
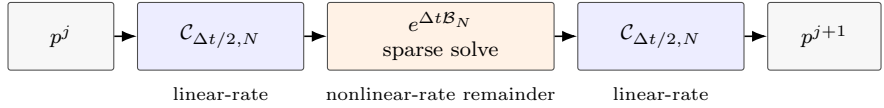

Closure operators, when restricted to a finite index window, fail to form a semigroup when a half-step creates components above $N$ that a later substep would then return to that index window.
We therefore compare with unprojected Strang splitting and choose timesteps small enough to bound every projection error.
Suppose we are interested in solving the system to a fixed time $T>0$. Conditional on this choice, we choose $\Delta t$ such that $J:=T/\Delta t$ is a positive integer.
For step $k\in\{0,\ldots,J-1\}$, let the nonnegative quantities $\delta_{\mathcal B,k}$ and $\delta_{\mathrm{pre},k}$ bound the tails outside the index window during the unprojected $\mathcal B$ substep and immediately before projection after the preceding linear-rate half-step, respectively.
Assume the stability and commutator bounds for Strang splitting on a weighted $\ell^1$ space~\cite{strangConstructionComparisonDifference1968,mclachlanSplittingMethods2002}.
Then a telescoping argument gives nonnegative constants $C_{\mathrm{split}}$, $C_{\mathcal A}$, and $C_{\mathcal B}$ such that
\begin{equation}
    \left\|\Pi_N p(T)-p^J\right\|_1
    \leq C_{\mathrm{split}}T\Delta t^2
    +C_{\mathcal B}T\max_{0\leq k<J}\delta_{\mathcal B,k}
    +C_{\mathcal A}T\max_{0\leq k<J}\delta_{\mathrm{pre},k}.
    \label{eq:strang_window_error}
\end{equation}
We obtain Eq.~\ref{eq:strang_window_error} by inserting the unprojected solution after each substep and telescoping over the $J$ steps.
Stability propagates each local splitting or projection defect, and summation gives the three terms proportional to $T$.
The three terms in Eq.~\ref{eq:strang_window_error} bound the splitting error, the remainder tail, and the tail discarded after a linear-rate half-step, respectively.
The bound omits numerical error in the two subsolvers; in the experiments, we choose their tolerances or step refinements below the reported splitting and index-window projection effects, except where we explicitly identify a reference-resolution floor.
The final term vanishes for a pure-death linear-rate half-step and, more generally, whenever the linear-rate flow remains inside the selected index window.
We estimate both index-window projection terms by cap refinement or comparison with an independently converged expensive large-cap reference.

\subsubsection{Richardson extrapolation}\label{sec:richardson}

When the index-window projection error lies below the time-discretization error, we cancel the leading $\Delta t^2$ commutator term by combining Strang runs with $J$ and $2J$ steps as
\begin{equation}
    p_{\mathrm{Rich},N}^J(T)
    =\tfrac{4}{3}\,\mathcal{S}_{\Delta t/2,N}^{2J}\bigl(\Pi_Np(0)\bigr)
    -\tfrac{1}{3}\,\mathcal{S}_{\Delta t,N}^{J}\bigl(\Pi_Np(0)\bigr).
    \label{eq:richardson}
\end{equation}
Eq.~\ref{eq:richardson} leaves an $\mathcal{O}(\Delta t^4)$ temporal term at the cost of about three times the coarse-run computation.
Unlike higher-order real compositions, Richardson requires no negative dissipative substeps~\cite{shengSolvingLinearPartial1989,suzukiGeneralTheoryFractal1991}.
The extrapolated vector need not remain nonnegative, so we use Richardson as an accuracy correction rather than as a Markov propagator.

\subsubsection{Splitting implementation}\label{sec:splitting_implementation}

For one count axis, we implement Eq.~\ref{eq:closure_arbitrary_input} as a matrix-vector multiplication with the stored half-step kernel.
We advance $\mathcal{B}$ with sparse-Jacobian Rosenbrock23~\cite{rackauckasDifferentialEquationsjlPerformantFeatureRich2017}.
For stationary hybrid Markov models, we seek the fixed point of the closure-Strang step map by power iteration.
If $\mathcal{A}=\bigoplus_{i=1}^K\mathcal{A}_i$, then we apply its propagator through $K$ mode-wise tensor contractions without materializing the full product-space coefficient matrix.
The computational cost is $\Theta((N+1)^{3K})$ for dense exponentiation and $\Theta(JK(N+1)^{K+1})$ for $J$ sets of $K$ mode contractions.
Thus the asymptotic ratio is
\begin{equation}
    \frac{\text{dense computational cost}}{\text{Kronecker--Strang computational cost}}
    = \Theta\!\left(\frac{(N+1)^{2K - 1}}{J\, K}\right).
    \label{eq:kronstrang_complexity_ratio}
\end{equation}
\subsection{Applications}\label{sec:applications}

For the four model families, we assign transitions to the closure part $\mathcal A$ and the remainder $\mathcal B$ as follows.

\paragraph{Branching}
Branching processes model populations in which individuals reproduce or die independently, as in cell-population models~\cite{athreyaBranchingProcesses1972a,kimmelBranchingProcessesBiology2015}.
In the scalar benchmark, each particle dies at rate $\mu$ or divides into two at rate $\lambda$.
Thus a population of size $n$ jumps to $n-1$ at rate $\mu n$ and to $n+1$ at rate $\lambda n$.
In the four-type model, a type-$i$ particle for $i\in\{1,\ldots,4\}$ dies, creates another type-$i$ particle, or creates a particle of the next type in a cycle, with each event occurring at a per-particle rate.
All rates are linear in the population counts, so we apply the closure of Section~\ref{sec:closure} to the complete forward generator and set $\mathcal{B}=0$.

\paragraph{Schl\"ogl kinetics}
The Schl\"ogl model describes the molecule count $X$ in a well-mixed reactor with two reversible reaction pairs~\cite{schloglChemicalReactionModels1972,vellelaStochasticDynamicsNonequilibrium2009}.
Let $V>0$ denote the reactor-volume parameter that scales the stochastic reaction propensities.
The pair $\varnothing\rightleftharpoons X$ gives state-independent production and per-molecule loss.
The pair $2X\rightleftharpoons3X$ produces nonlinear positive feedback and, in the parameter regime used below, two favored ranges of molecule counts.
We place $\varnothing\rightleftharpoons X$ in $\mathcal{A}$ and advance it by closure.
The quadratic and cubic propensities form a tridiagonal remainder $\mathcal{B}$, which we advance on a capped index window with the splitting method of Section~\ref{sec:splitting}.

\paragraph{Cyclic predator-prey}
Let $X_i$ denote the population of species $i\in\{1,\ldots,K\}$, and let $\boldsymbol{x}=(x_1,\ldots,x_K)\in\Z_{\geq0}^K$ denote a current population vector.
Each species receives individuals at a constant immigration rate and loses individuals at a constant per-capita death rate.
With species indices interpreted modulo $K$, the interaction $X_i+X_{i+1}\to2X_{i+1}$ removes one individual of species $i$ and adds one of species $i+1$ at rate $\gamma x_i x_{i+1}$, where $\gamma\geq0$ is the interaction-rate constant.
Let $\mathcal{L}_i^{\mathrm{lin}}$ denote the forward generator of the independent immigration-death process for species $i$.
For $a\in\{1,\ldots,K\}$, let $b=a+1$ modulo $K$.
We split the joint forward generator as
\begin{equation}
    \begin{aligned}
        \mathcal{L}
         & =\bigoplus_{i=1}^{K}\mathcal{L}_i^{\mathrm{lin}}+\mathcal{L}_{\mathrm{pred}}, \\
        \bigl(\mathcal{L}_{\mathrm{pred}}\bigr)_{(\ldots,x_a-1,x_b+1,\ldots),\,(\ldots,x_a,x_b,\ldots)}
         & =\gamma x_ax_b,
        \qquad x_a\geq1.
    \end{aligned}
    \label{eq:predprey_split}
\end{equation}
Column conservation sets $(\mathcal{L}_{\mathrm{pred}})_{\boldsymbol{x},\boldsymbol{x}}=-\gamma\sum_{a=1}^K x_a x_{a+1}$ with cyclic indexing.
The direct-sum term in Eq.~\ref{eq:predprey_split} forms $\mathcal{A}$, and $\mathcal{L}_{\mathrm{pred}}$ forms $\mathcal{B}$.
We advance $\mathcal{A}$ by the multivariate closure of Section~\ref{sec:closure} and $\mathcal{B}$ by a capped sparse solve inside the splitting method of Section~\ref{sec:splitting}.
Applying the linear-rate propagator in Eq.~\ref{eq:predprey_split} one coordinate at a time avoids materializing the coefficient matrix on $(N+1)^K$ states.

\paragraph{Telegraph and gene/RNA elongation}
The telegraph model represents stochastic gene expression by coupling an unbounded mRNA count $m$ to a gene that switches between inactive and active states~\cite{peccoudMarkovianModelingGeneProduct1995,shahrezaeiAnalyticalDistributionsStochastic2008}.
The gene/RNA (G/R) extension also resolves a finite chain of polymerase elongation stages.
Selected internal-state transitions produce an mRNA molecule, and each existing mRNA degrades independently.
Both models have the form of Eq.~\ref{eq:matrix_count_master}.
The matrix $\mathbf{B}$ is diagonal when production leaves the internal state unchanged and off-diagonal when production also advances the elongation chain.
We benchmark both the full finite-state closure of Section~\ref{sec:closure} and the pure-degradation split
\begin{equation}
    \mathcal{A} : \dot{\boldsymbol{P}}_m = \mu\bigl((m+1) \boldsymbol{P}_{m+1} - m \boldsymbol{P}_m\bigr),
    \qquad
    \mathcal{B} : \dot{\boldsymbol{P}}_m = \mathbf{A} \boldsymbol{P}_m + \mathbf{B} \boldsymbol{P}_{m-1},
    \qquad m\geq0.
    \label{eq:telegraph_split}
\end{equation}
In Eq.~\ref{eq:telegraph_split}, binomial thinning advances $\mathcal{A}$ exactly.
The coefficient operator $\mathcal{B}$ is lower triangular in $m$.
As in Figure~\ref{fig:dependency_structure}(b), this one-sided dependence eliminates return flux from above the requested index window.
We combine the two propagators by the splitting method of Section~\ref{sec:splitting}.
The supplement lists all numerical rates, initial conditions, caps, step counts, and reference settings.

\section{Results}\label{sec:numerics}

For transient problems, we compare closure and splitting with dense Pad\'e scaling-and-squaring for the matrix exponential~\cite{highamScalingSquaringMethod2009}, Krylov FSP~\cite{munskyFiniteStateProjection2006,saadAnalysisKrylovSubspace1992}, and uniformization~\cite{MarkoffChainsAid}.
The stationary comparisons additionally include PGF-FFT, block-Thomas elimination~\cite{neutsMatrixgeometricSolutionsStochastic1994}, sparse LU factorization~\cite{davisDirectMethodsSparse2006}, and dense SVD~\cite{MatrixComputations4th}.
Dense Pad\'e and dense SVD act on full capped coefficient matrices, while Krylov FSP, uniformization, sparse LU, and block-Thomas exploit matrix action, sparsity, or block structure; PGF-FFT works directly with the generating function.
We report warmed wall-clock times on an AMD Ryzen AI Max+ 395 workstation with 128\,GB RAM and $\ell^1$ differences computed on the stated index window relative to the stated reference.
For brevity, we call these differences errors; when the reference is numerical, they are not independently certified absolute errors and cannot resolve errors below the accuracy of the reference itself.
The benchmark sequence introduces a long tail, several count axes, nonlinear rates, and a finite internal state.
The supplement records all parameters, initial conditions, cap and step sweeps, solver tolerances, reference constructions, repetition counts, and figure-to-data mappings.

\subsection{Linear-rate multitype branching}\label{sec:exp_affine}

The scalar experiment isolates the effect of a long probability tail, whereas the four-type experiment exposes the $(N+1)^4$ growth of a conventionally truncated joint state space.
Figure~\ref{fig:ktype_branching} compares against the analytic geometric tail for $K=1$ and a direct-Cauchy-product closure reference for $K=4$.
The scalar closure marker evaluates the analytic formula in Eq.~\ref{eq:bd_geometric_tail}, so it does not time the general coefficient ODE.
For $K=4$, we integrate the reference coefficient ODE with relative tolerance $10^{-9}$ and absolute tolerance $10^{-12}$; because the closure solution itself is the reference, the figure does not measure its error independently.

\begin{figure}[!ht]
    \centering
    \includegraphics[width=0.98\textwidth]{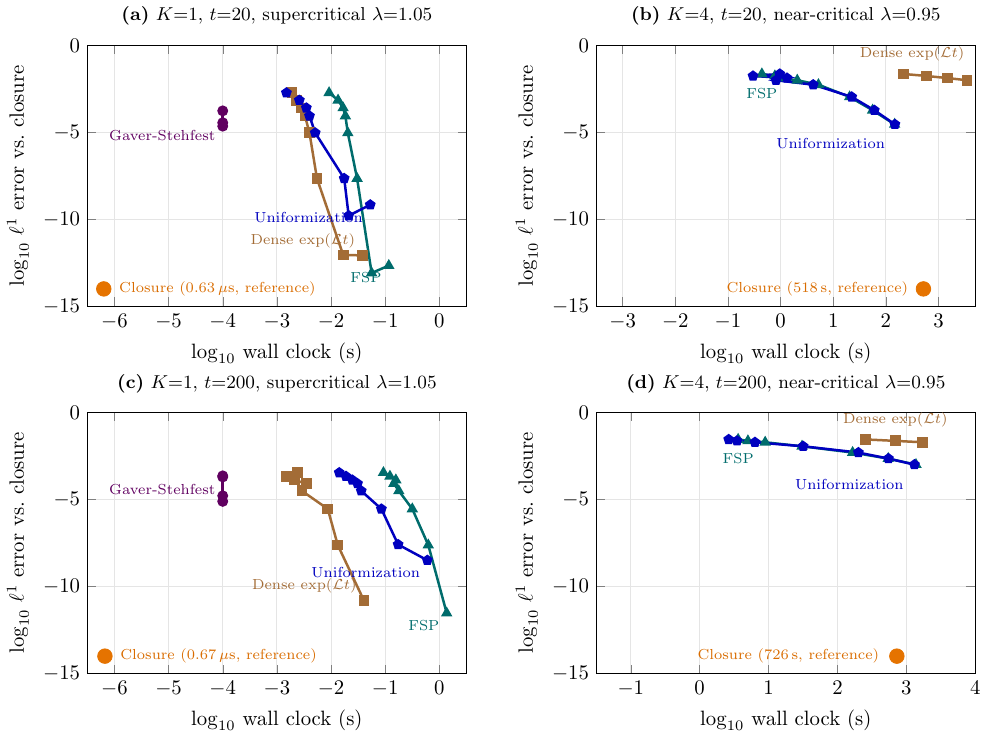}
    \caption{Wall clock versus $\ell^1$ error on the stated index window relative to the analytic ($K=1$) or closure ($K=4$) reference at $t=20$ (top) and $t=200$ (bottom).
        The closure markers report the reference runtimes and are placed near the lower plotting boundary because their errors relative to themselves are zero.
        Left: scalar birth-death on $0\leq n\leq100$; right: four-type cyclic cross-production on $[0,10]^4$.
        We sweep the truncation cap for dense Pad\'e, Krylov FSP, and uniformization; we apply Gaver-Stehfest only to the scalar case.}
    \label{fig:ktype_branching}
\end{figure}

At the four-type cap $N=12$, dense Pad\'e takes $2.4$-$2.8\times$ as much wall time as closure while retaining $1.3$-$2.0\times10^{-2}$ error on the reported index window.
FSP and uniformization run faster at these caps but retain the same boundary error.

\FloatBarrier
\subsection{Scalar nonlinear-rate splitting}\label{sec:exp_schlogl}

We use the bistable Schl\"ogl regime from Section~\ref{sec:applications} at two values of the reactor-volume parameter, $V\in\{25,500\}$.
Its quadratic and cubic propensities put the complete forward generator outside the closure class.
The experiment tests the split method on a sparse remainder.

\begin{figure}[!ht]
    \centering
    \includegraphics[width=0.95\textwidth]{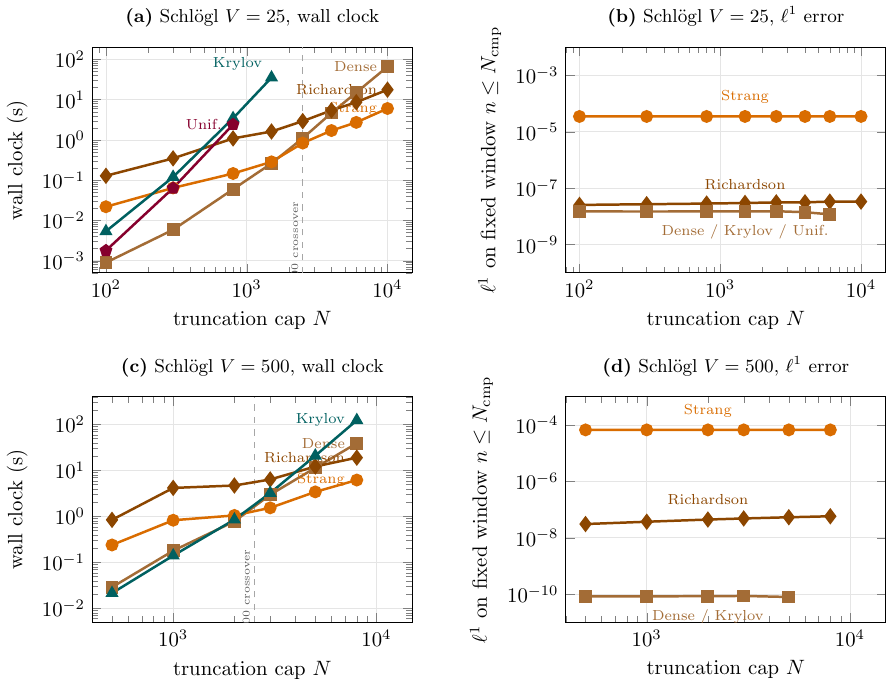}
    \caption{Schl\"ogl wall clock (left) and $\ell^1$ difference from the reference (right) versus cap $N$ at reactor-volume parameter $V=25$ (top) and $V=500$ (bottom).
        The right panels use the fixed index windows $0\leq n\leq50$ and $0\leq n\leq200$, relative to dense Pad\'e references at caps $10{,}000$ and $8{,}000$, respectively.
        We use 80 closure-Strang steps and combine the 80- and 160-step results for Richardson extrapolation.
        We compute dense Pad\'e, Krylov FSP, and uniformization on the full capped coefficient matrix.}
    \label{fig:schlogl_scaling}
\end{figure}

Closure-Strang crosses dense Pad\'e near $N=2500$.
At $V=500,N=8000$, we measure $6.2$\,s for closure-Strang, $39.2$\,s for dense Pad\'e, and $124$\,s for Krylov FSP.
At the largest cap for each volume, the 80-step Strang differences are $3.5\times10^{-5}$ and $6.7\times10^{-5}$, while the corresponding Richardson differences are $3.4\times10^{-8}$ and $5.8\times10^{-8}$.

\FloatBarrier
\subsection{Multi-species predator-prey with Kronecker-factored Strang}\label{sec:exp_predator_prey}

The two- and eight-species instances of Eq.~\ref{eq:predprey_split} test a sparse bilinear remainder on a product state space.
A cap $N$ on each of $K$ populations produces $(N+1)^K$ joint states even when $N$ is small.
\begin{figure}[!ht]
    \centering
    \includegraphics[width=0.98\textwidth]{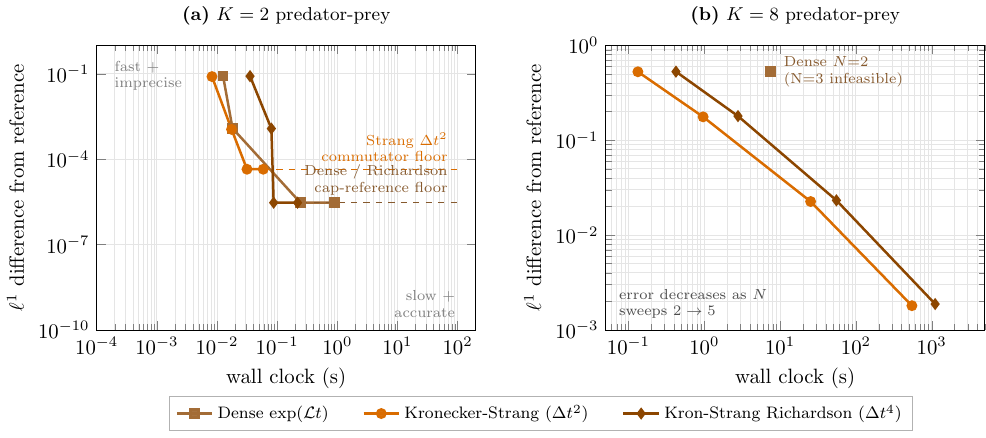}
    \caption{Wall clock versus $\ell^1$ difference from a refined Kronecker-Strang reference on the fixed index windows $[0,15]^2$ and $[0,2]^8$ for two- and eight-species predator-prey, respectively.
        \emph{Dense $\exp(\mathcal{L}t)$} denotes dense Pad\'e evaluation of the exponential of the full capped generator.
        \emph{Kronecker-Strang} applies the immigration-death half-steps by mode-wise tensor contractions and the bilinear predation step by sparse Krylov, while \emph{Kron-Strang Richardson} combines runs with step sizes $\Delta t$ and $\Delta t/2$ to cancel the leading splitting error; ``Kron'' abbreviates ``Kronecker.''
        The labels $\Delta t^2$ and $\Delta t^4$ give the leading global temporal-error orders of the two split traces.
        We sweep the per-species cap; dense Pad\'e is infeasible past $N=2$ for $K=8$.
        Richardson removes the leading commutator error but not the tail beyond the retained index window.}
    \label{fig:kspecies_predprey}
\end{figure}

At $K=2$, both factored Strang curves are faster than dense Pad\'e by joint dimension $961$.
For $N\geq40$, the same-cap difference from dense Pad\'e falls from $5.1\times10^{-5}$ for Strang to $7.9\times10^{-9}$ after Richardson, whereas the plotted fixed-window difference from the refined high-cap reference levels off near $3.0\times10^{-6}$.
At $K=8,N=2$, factored Strang and dense Pad\'e have fixed-window differences $0.527$ and $0.530$, respectively, while factored Strang is $56\times$ faster ($0.13$ versus $7.4$\,s).
At $N=3$, factored Strang takes $0.97$\,s while dense Pad\'e exceeds available memory; at $N=5$, factored Strang reaches a fixed-window difference of $1.8\times10^{-3}$ in $548$\,s.

\FloatBarrier
\subsection{Telegraph and G/R elongation chain}\label{sec:exp_telegraph}

We test the G/R model of Section~\ref{sec:applications} with $n_T\in\{6,14\}$ internal gene and elongation states.
After capping the mRNA count at $M$, the joint state has size $n_T(M+1)$, while the common degradation rate leaves one scalar characteristic for the entire chain.
The supplement gives the transition rates, time horizon, coefficient cap, step counts, and reference construction.

\begin{figure}[!ht]
    \centering
    \includegraphics[width=0.98\textwidth]{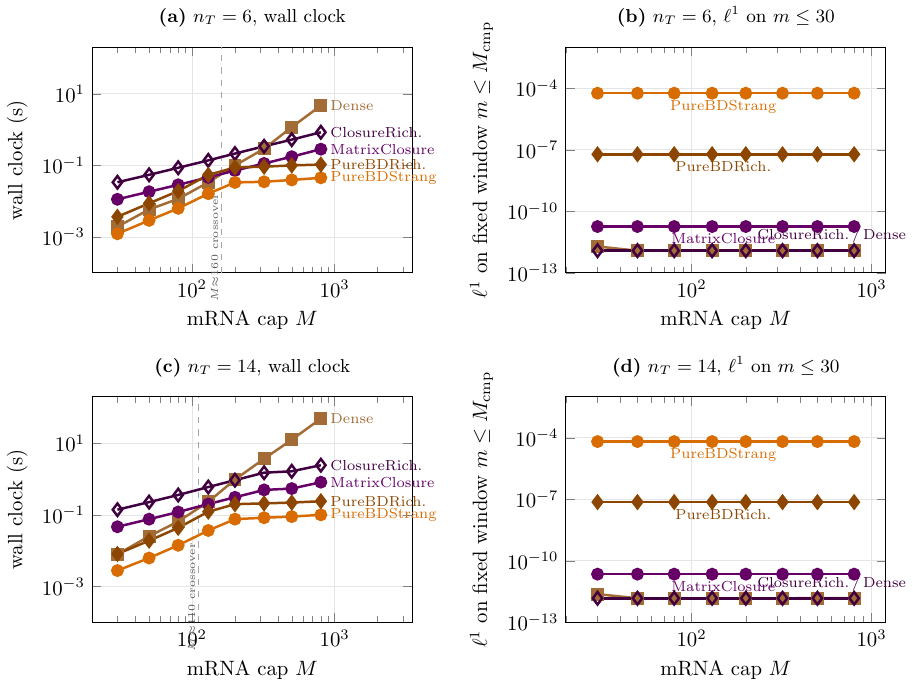}
    \caption{G/R elongation wall clock and count-index accuracy for $n_T=6$ (top) and $n_T=14$ (bottom).
        The left panels show wall clock versus the mRNA cap $M$; the right panels show the $\ell^1$ difference, summed over all internal states on $m\leq30$, from a 3200-step matrix-closure reference at cap $M_{\mathrm{ref}}=2000$.
        \emph{Dense} denotes dense Pad\'e evaluation of the capped matrix exponential; \emph{MatrixClosure} uses 1600 fourth-order Runge--Kutta steps, and \emph{ClosureRich.} extrapolates the 1600- and 3200-step closure solutions.
        \emph{PureBDStrang} uses 160 Strang steps with pure mRNA degradation as the linear-rate subproblem, and \emph{PureBDRich.} extrapolates the 160- and 320-step split solutions.
        Differences near $10^{-12}$ indicate the resolution of the numerical reference rather than a certified absolute error.}
    \label{fig:telegraph}
\end{figure}

Our matrix closure method is slower than dense Pad\'e at small caps but overcomes dense Pad\'e at $M\approx160$ for $n_T=6$ and $M\approx110$ for $n_T=14$.
At $M=800$, Richardson-refined matrix closure is $5.6\times$ and $20\times$ faster, respectively, with differences of $1.2\times10^{-12}$ and $1.5\times10^{-12}$ from the reference.
For $n_T=6$ at the same cap, pure-degradation Strang runs $104\times$ faster than dense Pad\'e with a difference of $5.8\times10^{-5}$, while Richardson reaches $6.1\times10^{-8}$ in $0.106$\,s.

\FloatBarrier
\subsection{Stationary solvers}\label{sec:exp_steady}

The rate decomposition also applies to solving for stationarity ($\mathcal{L}\mathbf{p}=0$).
The four panels cover a finite internal state, bilinear population coupling, several nonlinear reaction coordinates, and a fully linear-rate multitype model.

\begin{figure}[!ht]
    \centering
    \includegraphics[width=\textwidth]{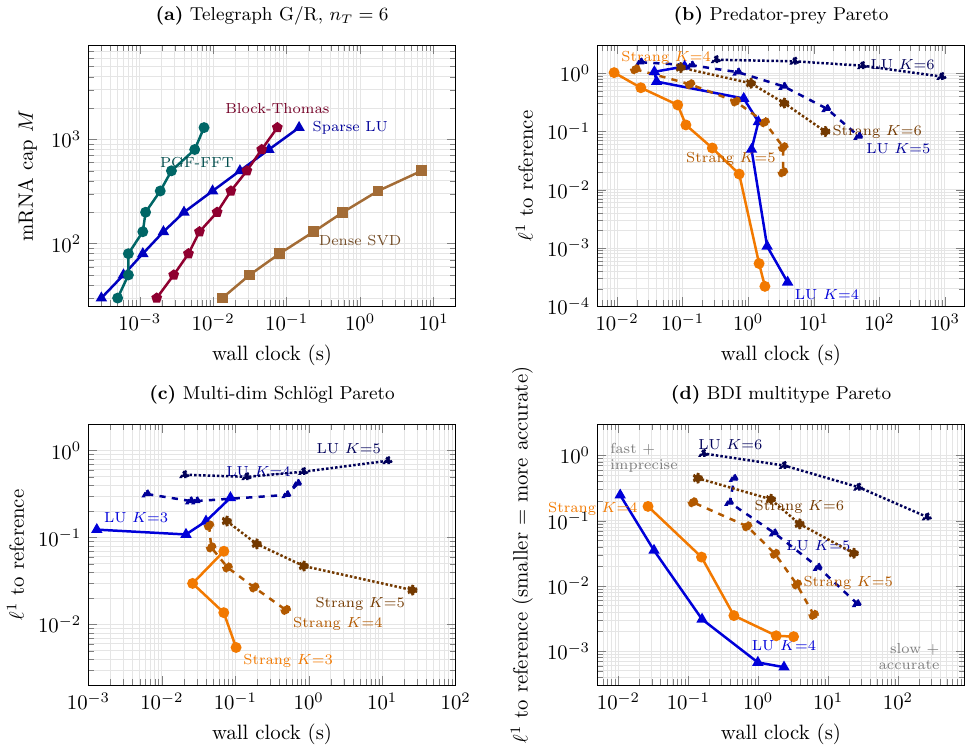}
    \caption{Stationary solver comparisons.
        Panel \textbf{(a)} compares telegraph G/R wall clock as the mRNA cap grows and does not imply that the methods have equal accuracy.
        Panels \textbf{(b)-(d)} plot wall clock against $\ell^1$ difference from the stated high-cap reference for closure-Strang power iteration and sparse LU applied to predator-prey, multidimensional Schl\"ogl, and multitype birth-death-immigration.}
    \label{fig:steady_state_combined}
\end{figure}

PGF-FFT has the lowest measured telegraph wall time from $M=80$ onward, but its $\ell^1$ difference from the high-cap reference is $7.7\times10^{-4}$ at $M=50$, $8.2\times10^{-2}$ at $M=130$, and $5.6\times10^{2}$ at $M=200$ because contour extraction amplifies roundoff.
Block-Thomas, by contrast, remains within $3\times10^{-16}$ of the reference for $M\geq50$ in this sweep.
For predator-prey at $K=4$ and target difference $5.5\times10^{-4}$, closure-Strang first reaches the target in $1.46$\,s, whereas sparse LU first reaches it in $3.99$\,s, so closure-Strang is $2.7\times$ faster.
At $K=5,N=8$, closure-Strang reaches a difference of $2.0\times10^{-2}$ in $3.4$\,s, while sparse LU reaches $8.2\times10^{-2}$ in $49$\,s.
The Schl\"ogl and birth-death-immigration panels also shift toward closure-Strang as $K$ grows, although their cap-error curves require careful accuracy matching.

\FloatBarrier
\section{Discussion}\label{sec:discussion}

For a fully linear-rate hierarchy, closure replaces $\Theta((N+1)^{3K})$ dense-exponential computational cost with $\Theta(SK(N+1)^{2K})$ computations per time step.
Closure also moves truncation from an evolving upper boundary to the initial data.
Dense Pad\'e remains faster for small scalar problems, but its time and storage grow more rapidly with the cap and number of count axes.

Splitting extends the method beyond the linear-rate subclass but reintroduces a cap for the remainder.
We retain boundary exactness only during the linear-rate substeps, and we preserve the sparse structure of $\mathcal{B}$ during the capped middle step.
Eq.~\ref{eq:strang_window_error} separates the second-order commutator error from the two projection tails; Richardson removes the leading commutator term but neither tail.

The favorable regime is therefore a large capped state space with a substantial linear-rate component and a sparse or absent remainder.
In this regime, closure avoids dense storage and artificial upper-boundary conditions during linear-rate evolution and can reduce wall time at matched accuracy.
Because the asymptotic bounds suppress implementation-dependent constants, they describe scaling but do not predict the finite-size crossover, which also depends on the model, tolerance, cap, and hardware.
For example, closure-Strang is $3.3\times$ faster than dense Pad\'e and $6.1\times$ faster than Krylov at $V=500,N=5000$, while only the factored method fits the eight-species problem beyond $N=2$ on our workstation.
State-specific matrix drift and higher-order splitting remain outside our assumptions~\cite{cuchieroAffineProcessesPositive2011,castellaSplittingMethodsComplex2009}.

\section*{Acknowledgments}

This research was supported in part by the Intramural Research Program of the National Institutes of Health (NIH); the contributions of the NIH author are considered Works of the United States Government, and the findings and conclusions are those of the author and do not necessarily reflect the views of the NIH or the U.S.\ Department of Health and Human Services.
The author thanks Carson Chow and Tom Chou for helpful comments.

\section*{Code and data availability}
The Supplemental Materials contain the Julia source code used to generate the results and the pinned package environment recorded in \texttt{Manifest.toml}.
The related \texttt{StochasticGene.jl} repository for fitting and analyzing stochastic gene-transcription models is available at \url{https://github.com/nih-niddk-mbs/StochasticGene.jl}.

\bibliographystyle{siam}
\bibliography{branching}

\end{document}